\documentclass[reqno]{amsart}
\usepackage[square]{natbib}
\newtheorem{thm}{Theorem}
\newtheorem{lem}[thm]{Lemma}
\newtheorem{cor}[thm]{Corollary}
\newtheorem{prop}[thm]{Proposition}
\theoremstyle{definition}

\newtheorem{ex}[thm]{Example}
\newtheorem{rem}[thm]{Remark}

\providecommand{\abs}[1]{\lvert#1\rvert}
\providecommand{\Abs}[1]{\Bigl\lvert#1\Bigr\rvert}

\begin{document}
%E PER PRIMA COSA SI METTONO TITOLO, AUTORI E VARIE DEL GENERE
\title[CLT for multicolor urns]{Central limit theorems for multicolor\\urns with dominated colors}
\author{Patrizia Berti}
\address{Patrizia Berti, Dipartimento di Matematica Pura ed Applicata ''G. Vitali'', Universita' di Modena e Reggio-Emilia, via Campi 213/B, 41100 Modena, Italy}
\email{patrizia.berti@unimore.it}
\author{Irene Crimaldi}
\address{Irene Crimaldi, Dipartimento di Matematica, Universita' di
Bologna, Piazza di Porta San Donato 5, 40126 Bologna, Italy}
\email{crimaldi@dm.unibo.it}
\author{Luca Pratelli}
\address{Luca Pratelli, Accademia Navale, viale Italia 72, 57100 Livorno,
Italy} \email{pratel@mail.dm.unipi.it}
\author{Pietro Rigo}
\address{Pietro Rigo (corresponding author), Dipartimento di Economia Politica e Metodi Quantitativi, Universita' di Pavia, via S. Felice 5, 27100 Pavia, Italy}
\email{prigo@eco.unipv.it} 

\date{\today}

%POI SI METTE L'ABSTRACT
\begin{abstract}
An urn contains balls of $d\geq 2$ colors. At each time $n\geq 1$, a
ball is drawn and then replaced together with a random number of
balls of the same color. Let ${\bf
A_n}=\,$diag$\bigl(A_{n,1},\ldots,A_{n,d}\bigr)$ be the $n$-th
reinforce matrix. Assuming $EA_{n,j}=EA_{n,1}$ for all $n$ and $j$,
a few CLT's are available for such urns. In real problems, however,
it is more reasonable to assume
\begin{gather*}EA_{n,j}=EA_{n,1}\quad\text{whenever }n\geq 1\text{ and
}1\leq j\leq d_0,
\\\liminf_nEA_{n,1}>\limsup_nEA_{n,j}\quad\text{whenever }j>d_0,
\end{gather*}
for some integer $1\leq d_0\leq d$. Under this condition, the usual
weak limit theorems may fail, but it is still possible to prove
CLT's for some slightly different random quantities. These random
quantities are obtained neglecting dominated colors, i.e., colors
from $d_0+1$ to $d$, and allow the same inference on the urn
structure. The sequence $({\bf A_n}:n\geq 1)$ is independent but
need not be identically distributed. Some statistical applications
are given as well.
\end{abstract}

\maketitle

\section{The problem}\label{prob}

An urn contains $a_j>0$ balls of color $j\in\{1,\ldots,d\}$ where
$d\geq 2$. At each time $n\geq 1$, a ball is drawn and then replaced
together with a random number of balls of the same color. Say that
$A_{n,j}\geq 0$ balls of color $j$ are added to the urn in case
$X_{n,j}=1$, where $X_{n,j}$ is the indicator of $\{$ball of color
$j$ at time $n\}$. Let
\begin{gather*}
N_{n,j}=a_j+\sum_{k=1}^nX_{k,j}A_{k,j}
\end{gather*}
be the number of balls of color $j$ in the urn at time $n$ and
\begin{gather*}
Z_{n,j}=\frac{N_{n,j}}{\sum_{i=1}^dN_{n,i}},\quad M_{n,j}=\frac{\sum_{k=1}^nX_{k,j}}{n}.
\end{gather*}

Fix $j$ and let $n\rightarrow\infty$. Then, under various
conditions, $Z_{n,j}\overset{a.s.}\longrightarrow Z_{(j)}$ for some
random variable $Z_{(j)}$. This typically implies
$M_{n,j}\overset{a.s.}\longrightarrow Z_{(j)}$. A CLT is available
as well. Define in fact
\begin{equation*}
C_{n,j}=\sqrt{n}\,\bigl(M_{n,j}-Z_{n,j}\bigr)\quad\text{and}\quad
D_{n,j}=\sqrt{n}\,\bigl(Z_{n,j}-Z_{(j)}\bigr).
\end{equation*}
As shown in \cite{BCPR}, under reasonable conditions one obtains
\begin{gather*}
(C_{n,j},\,D_{n,j})\longrightarrow\mathcal{N}(0,U_j)\times\mathcal{N}(0,V_j)\quad\text{stably}
\end{gather*}
for certain random variables $U_j$ and $V_j$. A nice consequence is
\begin{gather*}
\sqrt{n}\,\bigl(M_{n,j}-Z_{(j)}\bigr)=C_{n,j}+D_{n,j}\longrightarrow\mathcal{N}(0,U_j+V_j)\quad\text{stably}.
\end{gather*}
Stable convergence, in the sense of Aldous and Renyi, is a strong
form of convergence in distribution. The definition is recalled in
Section \ref{prel}.

For $(C_{n,j},\,D_{n,j})$ to converge, it is fundamental that
$EA_{n,j}=EA_{n,1}$ for all $n$ and $j$. In real problems, however,
it is more sound to assume that
\begin{gather*}
EA_{n,j}=EA_{n,1}\quad\text{whenever }n\geq 1\text{ and }1\leq j\leq
d_0,
\\\liminf_nEA_{n,1}>\limsup_nEA_{n,j}\quad\text{whenever }j>d_0,\notag
\end{gather*}
for some integer $1\leq d_0\leq d$. Roughly speaking, when $d_0<d$
some colors (those labelled from $d_0+1$ to $d$) are dominated by
the others. In this framework, for $j\in\{1,\ldots,d_0\}$,
meaningful quantities are
\begin{gather*}
C_{n,j}^*=\sqrt{n}\,\bigl(M_{n,j}^*-Z_{n,j}^*\bigr)\quad\text{and}\quad
D_{n,j}^*=\sqrt{n}\,\bigl(Z_{n,j}^*-Z_{(j)}\bigr)\quad\text{where}
\\M_{n,j}^*=\frac{\sum_{k=1}^nX_{k,j}}{1+\sum_{i=1}^{d_0}\sum_{k=1}^nX_{k,i}}\,,\quad
Z_{n,j}^*=\frac{N_{n,j}}{\sum_{i=1}^{d_0}N_{n,i}}.
\end{gather*}
If $d_0=d$, then $D_{n,j}^*=D_{n,j}$ and
$\abs{C_{n,j}^*-C_{n,j}}\leq\frac{1}{\sqrt{n}}$. If $d_0<d$, in a
sense, dealing with $(C_{n,j}^*,\,D_{n,j}^*)$ amounts to neglecting
dominated colors.

Our problem is to determine the limiting distribution of
$(C_{n,j}^*,\,D_{n,j}^*)$, under reasonable conditions, when
$d_0<d$.

\section{Motivations}\label{mot56fg}

Possibly, when $d_0<d$, $Z_{n,j}$ and $M_{n,j}$ have a more
transparent meaning than their counterparts $Z_{n,j}^*$ and
$M_{n,j}^*$. Accordingly, a CLT for $(C_{n,j},\,D_{n,j})$ is more
intriguing than a CLT for $(C_{n,j}^*,\,D_{n,j}^*)$. So, why dealing
with $(C_{n,j}^*,\,D_{n,j}^*)$ ?

The main reason is that $(C_{n,j},\,D_{n,j})$ merely fails to
converge in case
\begin{equation}\label{mf56gh}
\liminf_nEA_{n,j}>\frac{1}{2}\,\liminf_nEA_{n,1}\quad\text{for some
}j>d_0.
\end{equation}
Fix in fact $j\leq d_0$. Under some conditions,
$Z_{n,j}\overset{a.s.}\longrightarrow Z_{(j)}$ with $Z_{(j)}>0$
a.s.; see Lemma \ref{mf65tr}. Furthermore, condition \eqref{mf56gh}
yields
$\sqrt{n}\,\sum_{i=d_0+1}^dZ_{n,i}\overset{a.s.}\longrightarrow\infty$.
(This follows from Corollary 2 of \cite{MF} for $d=2$, but it can be
shown in general). Hence,
\begin{equation*}
D_{n,j}^*-D_{n,j}\geq
Z_{n,j}\,\sqrt{n}\sum_{i=d_0+1}^dZ_{n,i}\overset{a.s.}\longrightarrow\infty.\end{equation*}
Since $D_{n,j}^*$ converges stably, as proved in Theorem \ref{main},
$D_{n,j}$ fails to converge in distribution under \eqref{mf56gh}.

A CLT for $D_{n,j}$, thus, is generally not available. A way out
could be looking for the right norming factors, that is,
investigating whether $\frac{\alpha_n}{\sqrt{n}}\,D_{n,j}$ converges
stably for suitable constants $\alpha_n$. This is a reasonable
solution but we discarded it. In fact, as proved in Corollary
\ref{cor560ew3}, $(C_{n,j},\,D_{n,j})$ converges stably whenever
\begin{equation}\label{nuovcon45d}
\limsup_nEA_{n,j}<\frac{1}{2}\,\liminf_nEA_{n,1}\quad\text{for all
}j>d_0.\tag{1*}
\end{equation}
So, the choice of $\alpha_n$ depends on whether \eqref{mf56gh} or
\eqref{nuovcon45d} holds, and this is typically unknown in
applications (think to clinical trials). In addition, dealing with
$(C_{n,j}^*,\,D_{n,j}^*)$ looks natural (to us). Loosely speaking,
as the problem occurs because there are some dominated colors, the
trivial solution is just to neglect dominated colors.

A next point to be discussed is the practical utility (if any) of a
CLT for $(C_{n,j}^*,\,D_{n,j}^*)$ or $(C_{n,j},\,D_{n,j})$. To fix
ideas, we refer to $(C_{n,j}^*,\,D_{n,j}^*)$ but the same comments
apply to $(C_{n,j},\,D_{n,j})$ provided a CLT for the latter is
available. It is convenient to distinguish two situations. With
reference to a real problem, suppose the subset of {\em non
dominated} colors is some $J\subset\{1,\ldots,d\}$ and not
necessarily $\{1,\ldots,d_0\}$.

If $J$ is known, the main goal is to make inference on $Z_{(j)}$,
$j\in J$. To this end, the limiting distribution of $D_{n,j}^*$ is
useful. Knowing such distribution, for instance, asymptotic
confidence intervals for $Z_{(j)}$ are easily obtained. An example
(cf. Example \ref{ex1}) is given in Section \ref{m123a}.

But in various frameworks, $J$ is actually unknown (think to
clinical trials again). Then, the main focus is to identify $J$ and
the limiting distribution of $C_{n,j}^*$ can help. If such
distribution is known, the hypothesis \begin{equation*} H_0:J=J^*
\end{equation*} can be (asymptotically) tested for any
$J^*\subset\{1,\ldots,d\}$ with card$(J^*)\geq 2$. Details are in
Examples \ref{ex2} and \ref{ex3}.

A last remark is that our results become trivial for $d_0=1$. On one
hand, this is certainly a gap, as $d_0=1$ is important in
applications. On the other hand, $d_0=1$ is itself a trivial case.
Indeed, $Z_{(1)}=1$ a.s., so that no inference on $Z_{(1)}$ is
required.

This paper is the natural continuation of \cite{BCPR}. While the
latter deals with $d_0=d$, the present paper focus on $d_0<d$.
Indeed, our results hold for $d_0\leq d$, but they are contained in
Corollary 9 of \cite{BCPR} in the particular case $d_0=d$. In
addition to \cite{BCPR}, a few papers which inspired and affected
the present one are \cite{AMS} and \cite{MF}. Other related
references are \cite{BH}, \cite{BPR}, \cite{CLP}, \cite{J04},
\cite{J05}, \cite{MPS}, \cite{P}.

The paper is organized as follows. Section \ref{prel} recalls some
basic facts on stable convergence. Section \ref{m123a} includes the
main results (Theorem \ref{main} and Corollary \ref{cor560ew3}).
Precisely, conditions for
\begin{gather*}
(C_{n,j}^*,\,D_{n,j}^*)\longrightarrow\mathcal{N}(0,U_j)\times\mathcal{N}(0,V_j)\quad\text{stably
and}
\\(C_{n,j},\,D_{n,j})\longrightarrow\mathcal{N}(0,U_j)\times\mathcal{N}(0,V_j)\quad\text{stably
under \eqref{nuovcon45d}}
\end{gather*}
are given, $U_j$ and $V_j$ being the same random variables mentioned
in Section \ref{prob}. As a consequence,
\begin{gather*}
\sqrt{n}\,\bigl(M_{n,j}^*-Z_{(j)}\bigr)=C_{n,j}^*+D_{n,j}^*\longrightarrow\mathcal{N}(0,U_j+V_j)\quad\text{stably
and}
\\\sqrt{n}\,\bigl(M_{n,j}-Z_{(j)}\bigr)=C_{n,j}+D_{n,j}\longrightarrow\mathcal{N}(0,U_j+V_j)\quad\text{stably under \eqref{nuovcon45d}}.
\end{gather*}
Also, it is worth noting that $D_{n,j}^*$ and $D_{n,j}$ actually
converge in a certain stronger sense.

Finally, our proofs are admittedly long. To make the paper more
readable, they have been confined in Section \ref{proof56tvb8} and
in a final Appendix.

\section{Stable convergence}\label{prel}

Let $(\Omega,\mathcal{A},P)$ be a probability space and $S$ a metric
space. A {\em kernel} on $S$ (or a {\em random probability measure}
on $S$) is a measurable collection $N=\{N(\omega):\omega\in\Omega\}$
of probability measures on the Borel $\sigma$-field on $S$.
Measurability means that
\begin{equation*}
N(\omega)(f)=\int f(x)\,N(\omega)(dx)
\end{equation*}
is $\mathcal{A}$-measurable, as a function of $\omega\in\Omega$, for
each bounded Borel map $f:S\rightarrow\mathbb{R}$.

Let $(Y_n)$ be a sequence of $S$-valued random variables and $N$ a
kernel on $S$. Both $(Y_n)$ and $N$ are defined on
$(\Omega,\mathcal{A},P)$. Say that $Y_n$ converges {\it stably} to
$N$ in case
\begin{gather*}
P\bigl(Y_n\in\cdot\mid H\bigr)\longrightarrow E\bigl(N(\cdot)\mid
H\bigr)\quad\text{weakly}
\\\text{for all }H\in\mathcal{A}\text{ such that }P(H)>0.
\end{gather*}
Clearly, if $Y_n\rightarrow N$ stably, then $Y_n$ converges in
distribution to the probability law $E\bigl(N(\cdot)\bigr)$ (just
let $H=\Omega$). We refer to \cite{CLP} and references therein for
more on stable convergence. Here, we mention a strong form of stable
convergence, introduced in \cite{CLP}. Let
$\mathcal{F}=(\mathcal{F}_n)$ be any sequence of sub-$\sigma$-fields
of $\mathcal{A}$. Say that $Y_n$ converges $\mathcal{F}$-{\it stably
in strong sense} to $N$ in case
\begin{gather*}
E\bigl(f(Y_n)\mid\mathcal{F}_n\bigr)\overset{P}\longrightarrow
N(f)\quad\text{for all bounded continuous functions
}f:S\rightarrow\mathbb{R}.
\end{gather*}

Finally, we give two lemmas from \cite{BCPR}. In both,
$\mathcal{G}=(\mathcal{G}_n)$ is an increasing filtration. Given
kernels $M$ and $N$ on $S$, let $M\times N$ denote the kernel on
$S\times S$ defined as
\begin{equation*} \bigl(M\times N\bigr)(\omega)=M(\omega)\times
N(\omega)\quad\text{for all }\omega\in\Omega. \end{equation*}

\begin{lem}\label{plm}
Let $Y_n$ and $Z_n$ be $S$-valued random variables and $M$ and $N$
kernels on $S$, where $S$ is a separable metric space. Suppose
$\sigma(Y_n)\subset\mathcal{G}_n$ and
$\sigma(Z_n)\subset\mathcal{G}_\infty$ for all $n$, where
$\mathcal{G}_\infty=\sigma(\cup_n\mathcal{G}_n)$. Then,
\begin{equation*}
(Y_n,Z_n)\longrightarrow M\times N\quad\text{stably}
\end{equation*}
provided $Y_n\rightarrow M$ stably and $Z_n\rightarrow N$
$\mathcal{G}$-stably in strong sense.
\end{lem}

\begin{lem}\label{hftuimn}
Let $(Y_n)$ be a $\mathcal{G}$-adapted sequence of real random
variables. If $\sum_{n=1}^\infty\frac{EY_n^2}{n^2}<\infty$ and
$E\bigl(Y_{n+1}\mid\mathcal{G}_n\bigr)\overset{a.s.}\longrightarrow
Y$, for some random variable $Y$, then
\begin{equation*}
n\sum_{k\geq n}\frac{Y_k}{k^2}\,\overset{a.s.}\longrightarrow
Y\quad\text{and}\quad\frac{1}{n}\sum_{k=1}^nY_k\overset{a.s.}\longrightarrow
Y.
\end{equation*}
\end{lem}

\section{Main results}\label{m123a}

In the sequel, $X_{n,j}$ and $A_{n,j}$, $n\geq 1$, $1\leq j\leq d$,
are real random variables on the probability space
$(\Omega,\mathcal{A},P)$ and $\mathcal{G}=(\mathcal{G}_n:n\geq 0)$,
where
\begin{equation*}
\mathcal{G}_0=\{\emptyset,\Omega\},\quad\mathcal{G}_n=\sigma\bigl(X_{k,j},\,A_{k,j}:1\leq
k\leq n,\,1\leq j\leq d\bigr).
\end{equation*}
Let $N_{n,j}=a_j+\sum_{k=1}^nX_{k,j}A_{k,j}$ where $a_j>0$ is a constant. We assume that
\begin{gather}\label{indip}
X_{n,j}\in\{0,1\},\quad\sum_{j=1}^dX_{n,j}=1,\quad 0\leq
A_{n,j}\leq\beta\quad\text{for some constant }\beta,
\\\bigl(A_{n,j}:1\leq j\leq
d\bigr)\,\text{ independent of
}\,\mathcal{G}_{n-1}\vee\sigma\bigl(X_{n,j}:1\leq j\leq
d\bigr),\notag
\\
Z_{n,j}=P\bigl(X_{n+1,j}=1\mid\mathcal{G}_n\bigr)=\frac{N_{n,j}}{\sum_{i=1}^dN_{n,i}}\,\,\text{
a.s.}.\notag
\end{gather}
Given an integer $1\leq d_0\leq d$, let us define
\begin{equation*}
\lambda_0=0\,\text{ if }\,d_0=d\,\text{ and
}\,\lambda_0=\max_{d_0<j\leq d}\limsup_nEA_{n,j}\,\,\text{ if
}\,d_0<d.
\end{equation*}
We also assume that
\begin{gather}\label{hg5dr}
EA_{n,j}=EA_{n,1}\quad\text{for }n\geq 1\text{ and }1\leq j\leq d_0,
\\m:=\lim_nEA_{n,1},\quad m>\lambda_0,\quad q_j:=\lim_nEA_{n,j}^2\quad\text{for }1\leq j\leq d_0.\notag
\end{gather}

A few useful consequences are collected in the following lemma.
Define
\begin{gather*}
S_n^*=\sum_{i=1}^{d_0}N_{n,i}\quad\text{and}\quad
S_n=\sum_{i=1}^dN_{n,i}.
\end{gather*}

\begin{lem}\label{mf65tr} Under conditions \eqref{indip}-\eqref{hg5dr}, as $n\rightarrow\infty$,
\begin{gather*}
\frac{S_n^*}{n}\overset{a.s.}\longrightarrow m\quad\text{and}\quad
\frac{S_n}{n}\overset{a.s.}\longrightarrow m,
\\n^{1-\lambda}\sum_{i=d_0+1}^dZ_{n,i}\overset{a.s.}\longrightarrow
0\quad\text{ whenever }d_0<d\text{ and }\lambda>\frac{\lambda_0}{m},
\\Z_{n,j}\overset{a.s.}\longrightarrow Z_{(j)}\quad\text{for each
}1\leq j\leq d_0,
\end{gather*}
where each $Z_{(j)}$ is a random variable such that $Z_{(j)}>0$
a.s..
\end{lem}

For $d=2$, Lemma \ref{mf65tr} follows from results in \cite{MF} and
\cite{MPS}. For arbitrary $d$, it is possibly known but we do not
know of any reference. Accordingly, a proof of Lemma \ref{mf65tr} is
given in the Appendix. We also note that, apart from a few
particular cases, the probability distribution of $Z_{(j)}$ is not
known (even if $d_0=d$).

We aim to settle the asymptotic behavior of
\begin{gather*}
C_{n,j}=\sqrt{n}\,\bigl(M_{n,j}-Z_{n,j}\bigr),\quad
D_{n,j}=\sqrt{n}\,\bigl(Z_{n,j}-Z_{(j)}\bigr),
\\C_{n,j}^*=\sqrt{n}\,\bigl(M_{n,j}^*-Z_{n,j}^*\bigr),\quad
D_{n,j}^*=\sqrt{n}\,\bigl(Z_{n,j}^*-Z_{(j)}\bigr),
\end{gather*}
where $j\in\{1,\ldots,d_0\}$ and
\begin{gather*}
M_{n,j}=\frac{\sum_{k=1}^nX_{k,j}}{n},\quad
M_{n,j}^*=\frac{\sum_{k=1}^nX_{k,j}}{1+\sum_{k=1}^n\sum_{i=1}^{d_0}X_{k,i}},
\quad
Z_{n,j}^*=\frac{N_{n,j}}{\sum_{i=1}^{d_0}N_{n,i}}.
\end{gather*}

Let $\mathcal{N}(a,b)$ denote the one-dimensional Gaussian law with
mean $a$ and variance $b\geq 0$ (where $\mathcal{N}(a,0)=\delta_a$).
Note that $\mathcal{N}(0,L)$ is a kernel on $\mathbb{R}$ for each
real non negative random variable $L$. We are in a position to state
our main result.

\begin{thm}\label{main} If conditions
\eqref{indip}-\eqref{hg5dr} hold, then
\begin{gather*}
C_{n,j}^*\longrightarrow \mathcal{N}(0,U_j)\text{ stably and }
\\D_{n,j}^*\longrightarrow \mathcal{N}(0,V_j)\,\,\,\,\mathcal{G}\text{-stably in strong
sense}
\\\text{for each }\,j\in\{1,\ldots,d_0\},\text{ where }\,U_j=V_j-Z_{(j)}(1-Z_{(j)})
\\\text{and }\,V_j=\frac{Z_{(j)}}{m^2}\,\bigl\{\,q_j\,(1-Z_{(j)})^2\,+\,Z_{(j)}\sum_{i\leq
d_0,i\neq j}q_i\,Z_{(i)}\,\bigr\}.
\end{gather*}
In particular (by Lemma \ref{plm}),
\begin{gather*}
(C_{n,j}^*,\,D_{n,j}^*)\longrightarrow\mathcal{N}(0,U_j)\times\mathcal{N}(0,V_j)
\text{ stably}.
\end{gather*}
\end{thm}

As noted in Section \ref{mot56fg}, Theorem \ref{main} has been
thought for the case $d_0<d$, and it reduces to Corollary 9 of
\cite{BCPR} in the particular case $d_0=d$. We also remark that some
assumptions can be stated in a different form. In particular, under
suitable extra conditions, Theorem \ref{main} works even if
$(A_{n,1},\ldots,A_{n,d})$ independent of
$\mathcal{G}_{n-1}\vee\sigma(X_{n,1},\ldots,X_{n,d})$ is weakened
into
\begin{equation*} (A_{n,1},\ldots,A_{n,d})\,\,\text{ conditionally
independent of }\,\,(X_{n,1},\ldots,X_{n,d})\,\,\text{ given
}\,\,\mathcal{G}_{n-1};\end{equation*} see Remark 8 of \cite{BCPR}.

The proof of Theorem \ref{main} is deferred to Section
\ref{proof56tvb8}. Here, we stress a few of its consequences.

We already know (from Section \ref{mot56fg}) that
$(C_{n,j},\,D_{n,j})$ may fail to converge when $d_0<d$. There is a
remarkable exception, however.

\begin{cor}\label{cor560ew3}
Under conditions \eqref{indip}-\eqref{hg5dr}, if $2\,\lambda_0<m$
(that is, \eqref{nuovcon45d} holds) then
\begin{gather*}
C_{n,j}\longrightarrow \mathcal{N}(0,U_j)\text{ stably and }\,
D_{n,j}\longrightarrow
\mathcal{N}(0,V_j)\,\,\,\,\mathcal{G}\text{-stably in strong sense}
\end{gather*}
for each $j\in\{1,\ldots,d_0\}$. In particular (by Lemma \ref{plm}),
\begin{gather*}
(C_{n,j},\,D_{n,j})\longrightarrow\mathcal{N}(0,U_j)\times\mathcal{N}(0,V_j)
\text{ stably}.
\end{gather*}
\end{cor}
\begin{proof}
By Theorem \ref{main}, it is enough to prove
$D_{n,j}^*-D_{n,j}\overset{P}\longrightarrow 0$ and
$C_{n,j}^*-C_{n,j}\overset{P}\longrightarrow 0$. It can be assumed
$d_0<d$. Note that
\begin{gather*}
\Abs{D_{n,j}^*-D_{n,j}}=\sqrt{n}\,Z_{n,j}\,\bigl(\frac{S_n}{S_n^*}-1\bigr)\leq\frac{S_n}{S_n^*}\,\,\sqrt{n}\sum_{i=d_0+1}^dZ_{n,i},
\\C_{n,j}^*-C_{n,j}=D_{n,j}-D_{n,j}^*\,+\,M_{n,j}\,\sqrt{n}\,\,\frac{\sum_{i=d_0+1}^dM_{n,i}-\frac{1}{n}}{\frac{1}{n}+\sum_{i=1}^{d_0}M_{n,i}}.
\end{gather*}
By Lemma \ref{mf65tr} and $2\,\lambda_0<m$, there is
$\alpha>\frac{1}{2}$ such that
$n^{\alpha}\sum_{i=d_0+1}^dZ_{n,i}\overset{a.s.}\longrightarrow 0$.
Thus, it remains only to see that
$\sqrt{n}\,M_{n,i}\overset{a.s.}\longrightarrow 0$ for each $i>d_0$.
Fix $i>d_0$ and define
$L_{n,i}=\sum_{k=1}^n\frac{X_{k,i}-Z_{k-1,i}}{\sqrt{k}}$. Since
$(L_{n,i}:n\geq 1)$ is a $\mathcal{G}$-martingale and
\begin{equation*}
\sum_nE\bigl\{(L_{n+1,i}-L_{n,i})^2\mid\mathcal{G}_n\bigr\}=\sum_n\frac{Z_{n,i}(1-Z_{n,i})}{n+1}\leq
\sum_n\frac{n^\alpha Z_{n,i}}{n^{1+\alpha}}<\infty\quad\text{a.s.},
\end{equation*}
then $L_{n,i}$ converges a.s.. By Kronecker lemma,
\begin{equation*}
\frac{1}{\sqrt{n}}\,\sum_{k=1}^{n}(X_{k,i}-Z_{k-1,i})=\frac{1}{\sqrt{n}}\,\sum_{k=1}^{n}\sqrt{k}\,\frac{X_{k,i}-Z_{k-1,i}}{\sqrt{k}}\,\overset{a.s.}\longrightarrow\,
0.
\end{equation*}
Since $\frac{1}{\sqrt{n}}\,\sum_{k=1}^{n}k^{-\alpha}\longrightarrow
0$ and $Z_{k,i}=\,$o$(k^{-\alpha})$ a.s., it follows that
\begin{gather*}
\sqrt{n}\,M_{n,i}=\frac{1}{\sqrt{n}}\,\sum_{k=1}^n(X_{k,i}-Z_{k-1,i})\,+\,\frac{1}{\sqrt{n}}\,\sum_{k=0}^{n-1}Z_{k,i}\,
\overset{a.s.}\longrightarrow\, 0.
\end{gather*}
\end{proof}

Theorem \ref{main} has some statistical implications as well.

\begin{ex}\label{ex1} {\bf (A statistical use of
$D_{n,j}^*$).} Suppose $d_0>1$, conditions
\eqref{indip}-\eqref{hg5dr} hold, and fix $j\leq d_0$. Let
$(V_{n,j}:n\geq 1)$ be a sequence of consistent estimators of $V_j$,
that is, $V_{n,j}\overset{P}\longrightarrow V_j$ and
$\sigma(V_{n,j})\subset\mathcal{D}_n$ for each $n$ where
\begin{equation*}
\mathcal{D}_n=\sigma\bigl(X_{k,i}A_{k,i},\,X_{k,i}:1\leq k\leq
n,\,1\leq i\leq d\bigr)
\end{equation*}
is the $\sigma$-field corresponding to the "available data". Since
$(V_{n,j})$ is $\mathcal{G}$-adapted, Theorem \ref{main} yields
\begin{equation*}
(D_{n,j}^*,\,V_{n,j})\longrightarrow\mathcal{N}(0,V_j)\times\delta_{V_j}\quad\mathcal{G}\text{-stably
in strong sense}.
\end{equation*}
Since $d_0>1$, then $0<Z_{(j)}<1$ a.s., or equivalently $V_j>0$
a.s.. Hence,
\begin{gather*}
I_{\{V_{n,j}>0\}}\,\frac{D_{n,j}^*}{\sqrt{V_{n,j}}}\longrightarrow\mathcal{N}(0,1)\quad\mathcal{G}\text{-stably
in strong sense}.
\end{gather*}
For large $n$, this fact allows to make inference on $Z_{(j)}$. For
instance,
\begin{equation*}
Z_{n,j}^*\pm\frac{u_\alpha}{\sqrt{n}}\,\sqrt{V_{n,j}}
\end{equation*}
provides an asymptotic confidence interval for $Z_{(j)}$ with
(approximate) level $1-\alpha$, where $u_\alpha$ is such that
$\mathcal{N}(0,1)(u_\alpha,\,\infty)=\frac{\alpha}{2}$.

An obvious consistent estimator of $V_j$ is
\begin{gather*}
V_{n,j}=\frac{1}{m_n^2}\,\bigl\{\,Q_{n,j}\,(1-Z_{n,j})^2\,+\,Z_{n,j}^2\sum_{i\leq
d_0,i\neq j}Q_{n,i}\,\bigr\}\quad\text{where}
\\m_n=\frac{\sum_{k=1}^n\sum_{i=1}^dX_{k,i}A_{k,i}}{n}\,\text{
and }\,Q_{n,i}=\frac{\sum_{k=1}^nX_{k,i}A_{k,i}^2}{n}.\notag
\end{gather*}
In fact,
$E(X_{n+1,i}A_{n+1,i}^2\mid\mathcal{G}_n)=Z_{n,i}\,EA_{n+,i}^2\overset{a.s.}\longrightarrow
Z_{(i)}\,q_i$ for all $i\leq d_0$, so that Lemma \ref{hftuimn}
implies $Q_{n,i}\overset{a.s.}\longrightarrow Z_{(i)}\,q_i$.
Similarly, $m_n\overset{a.s.}\longrightarrow m$. Therefore,
$V_{n,j}\overset{a.s.}\longrightarrow V_j$.

Finally, Theorem \ref{main} also implies
$\sqrt{n}\,\bigl(M_{n,j}^*-Z_{(j)}\bigr)=C_{n,j}^*+D_{n,j}^*\longrightarrow\mathcal{N}(0,U_j+V_j)$
stably. So, another asymptotic confidence interval for $Z_{(j)}$ is
$M_{n,j}^*\pm\frac{u_\alpha}{\sqrt{n}}\,\sqrt{G_{n,j}}$ where
$G_{n,j}$ is a consistent estimator of $U_j+V_j$. One merit of the
latter interval is that it does not depend on the initial
composition $a_i$, $i=1,\ldots,d_0$ (provided this is true for
$G_{n,j}$ as well).

\end{ex}

\begin{ex}\label{ex2} {\bf (A statistical use of
$C_{n,j}^*$).} Suppose
\begin{equation*}
EA_{n,j}=\mu_j\,\text{ and
}\,\text{var}(A_{n,j})=\sigma_j^2>0\,\text{ for all }n\geq 1\text{
and }1\leq j\leq d.
\end{equation*}
Suppose also that conditions \eqref{indip}-\eqref{hg5dr} hold with
some $J\subset\{1,\ldots,d\}$ in the place of $\{1,\ldots,d_0\}$,
where card$(J)>1$, that is
\begin{gather*}
\mu_r=m>\mu_s\quad\text{whenever }r\in J\text{ and }s\notin J.
\end{gather*}
Both $J$ and card$(J)$ are unknown, and we aim to test the
hypothesis $H_0:J=J^*$ where $J^*\subset\{1,\ldots,d\}$ and
card$(J^*)>1$. Note that $U_j$ can be written as
\begin{equation*}
U_j=\frac{Z_{(j)}}{m^2}\,\bigl\{(1-Z_{(j)})^2\sigma_j^2+Z_{(j)}\sum_{i\in
J,i\neq j}Z_{(i)}\,\sigma_i^2\bigr\},\quad j\in J.
\end{equation*}
Fix $j\in J^*$. Under $H_0$, a consistent estimator of $U_j$ is
\begin{gather*}
U_{n,j}=\frac{Z_{n,j}}{\widehat{m}_n^2}\,\bigl\{(1-Z_{n,j})^2\widehat{\sigma}_{n,j}^2+Z_{n,j}\sum_{i\in
J^*,i\neq
j}Z_{n,i}\,\widehat{\sigma}_{n,i}^2\bigr\}\quad\text{where}
\\\widehat{m}_n=\frac{1}{\text{card}(J^*)}\sum_{i\in J^*}\widehat{m}_{n,i},\,\,
\widehat{m}_{n,i}=\frac{\sum_{k=1}^nX_{k,i}A_{k,i}}{\sum_{k=1}^nX_{k,i}},\,\,\widehat{\sigma}_{n,i}^2=\frac{\sum_{k=1}^nX_{k,i}\bigl(A_{k,i}-\widehat{m}_{n,i})^2}{\sum_{k=1}^nX_{k,i}}.
\end{gather*}
Note that $\sum_{k=1}^nX_{k,i}>0$, eventually a.s., so that
$\widehat{m}_{n,i}$ and $\widehat{\sigma}_{n,i}^2$ are well defined.
Similarly $\widehat{m}_n>0$, eventually a.s., so that $U_{n,j}$ is
well defined. Next, defining $C_{n,j}^*$ in the obvious way (i.e.,
with $J^*$ in the place of $\{1,\ldots,d_0\}$), Theorem \ref{main}
implies
\begin{gather*}
K_{n,j}:=I_{\{U_{n,j}>0\}}\,\frac{C_{n,j}^*}{\sqrt{U_{n,j}}}\longrightarrow\mathcal{N}(0,1)\quad\text{stably
under }H_0.
\end{gather*}
The converse is true as well, i.e., $K_{n,j}$ fails to converge in
distribution to $\mathcal{N}(0,1)$ when $H_0$ is false. (This can be
proved arguing as in Remark \ref{h6t98d34s}; we omit a formal
proof). Thus, an asymptotic critical region for $H_0$, with
approximate level $\alpha$, is $\bigl\{\abs{K_{n,j}}\geq
u_\alpha\bigr\}$ with $u_\alpha$ satisfying
$\mathcal{N}(0,1)(u_\alpha,\,\infty)=\frac{\alpha}{2}$. In real
problems, sometimes, it is known in advance that $j_0\in J$ for some
$j_0\in J^*$. Then, $j=j_0$ is a natural choice in the previous
test. Otherwise, an alternative option is a critical region of the
type $\bigcup_{i\in J^*}\bigl\{\abs{K_{n,i}}\geq u_i\bigr\}$ for
suitable $u_i$. This results in a more powerful test but requires
the joint limit distribution of $\bigl(K_{n,i}:i\in J^*\bigr)$ under
$H_0$. Such a distribution is given in \cite{BCPR} when
$J^*=\{1,\ldots,d\}$, and can be easily obtained for arbitrary $J^*$
using the techniques of this paper.

\end{ex}

\begin{ex}\label{ex3} {\bf (Another statistical use of
$C_{n,j}^*$).} As in Example \ref{ex2} (and under the same
assumptions), we aim to test $H_0:J=J^*$. Contrary to Example
\ref{ex2}, however, we are given observations $A_{k,j}$, $1\leq
k\leq n$, $1\leq j\leq d$, but no urn is explicitly assigned. This
is a main problem in statistical inference, usually faced by the
ANOVA techniques and their very many ramifications. A solution to
this problem is using $C_{n,j}^*$, as in Example \ref{ex2}, after
{\em simulating }the $X_{n,j}$. The simulation is not hard. Take in
fact an i.i.d. sequence $(Y_n:n\geq 0)$, independent of the
$A_{k,j}$, with $Y_0$ uniformly distributed on $(0,1)$. Let $a_i=1$,
$Z_{0,i}=\frac{1}{d}$ for $i=1,\dots,d$, and
\begin{equation*}
X_{1,j}=I_{\{F_{0,j-1}<Y_0\leq F_{0,j}\}}\quad\text{where
}\,F_{0,j}=\sum_{i=1}^jZ_{0,i}\text{ and }F_{0,0}=0.
\end{equation*}
By induction, for each $n\geq 1$,
\begin{gather*}
X_{n+1,j}=I_{\{F_{n,j-1}<Y_n\leq F_{n,j}\}}\quad\text{where
}\,F_{n,j}=\sum_{i=1}^jZ_{n,i},
\\F_{n,0}=0\,\text{ and }\,Z_{n,i}=\frac{1+\sum_{k=1}^nX_{k,i}A_{k,i}}{d+\sum_{r=1}^d\sum_{k=1}^nX_{k,r}A_{k,r}}.
\end{gather*}
Now, $H_0$ can be asymptotically tested as in Example \ref{ex2}. In
addition, since $A_{k,i}$ is actually observed (unlike Example
\ref{ex2}, where only $X_{k,i}A_{k,i}$ is observed),
$\widehat{m}_{n,i}$ and $\widehat{\sigma}_{n,i}^2$ can be taken as
\begin{gather*}
\widehat{m}_{n,i}=\frac{\sum_{k=1}^nA_{k,i}}{n}\quad\text{and}\quad
\widehat{\sigma}_{n,i}^2=\frac{\sum_{k=1}^n\bigl(A_{k,i}-\widehat{m}_{n,i})^2}{n}.
\end{gather*}

Clearly, this procedure needs to be much developed and investigated.
By now, however, it looks (to us) potentially fruitful.

\end{ex}

\section{Proof of Theorem \ref{main}}\label{proof56tvb8}

Next result, of possible independent interest, is inspired by ideas
in \cite{BCPR} and \cite{CLP}.

\begin{prop}\label{beo5t}
Let $\mathcal{F}=(\mathcal{F}_n)$ be an increasing filtration and
$(Y_n)$ an $\mathcal{F}$-adapted sequence of real integrable random
variables. Suppose $Y_n\overset{a.s.}\longrightarrow Y$ for some
random variable $Y$ and $H_n\in\mathcal{F}_n$ are events satisfying
$P(H_n^c$ i.o.$)=0$. Then,
\begin{gather*}
\sqrt{n}\,(Y_n-Y)\longrightarrow\mathcal{N}(0,V)\quad\mathcal{F}\text{-stably
in strong sense},
\end{gather*}
for some random variable $V$, whenever
\begin{equation}\label{basicnew}
E\bigl\{I_{H_n}\,\bigl(E(Y_{n+1}\mid\mathcal{F}_n)-Y_n\bigr)^2\bigr\}=\text{o}(n^{-3}),
\end{equation}
\begin{equation}\label{f6y8j0w}
\sqrt{n}\,E\bigl\{I_{H_n}\,\sup_{k\geq
n}\,\abs{E(Y_{k+1}\mid\mathcal{F}_k)-Y_{k+1}}\bigr\}\longrightarrow
0,
\end{equation}
\begin{equation}\label{udlq2e}
n\,\sum_{k\geq n}(Y_k-Y_{k+1})^2\overset{P}\longrightarrow V.
\end{equation}
\end{prop}

\begin{proof}
We base on the following result, which is a consequence of Corollary
7 of \cite{CLP}. {\em Let $(L_n)$ be an $\mathcal{F}$-martingale
such that $L_n\overset{a.s.}\longrightarrow L$. Then,
$\sqrt{n}\,(L_n-L)\longrightarrow\mathcal{N}(0,V)$
$\mathcal{F}$-stably in strong sense whenever}
\begin{equation*}
\text{{\bf (i)}}\quad\lim_n\sqrt{n}\,E\bigl\{I_{H_n}\,\sup_{k\geq
n}\,\abs{L_k-L_{k+1}}\bigr\}=0;\quad \text{{\bf (ii)}}\quad
n\sum_{k\geq n}(L_k-L_{k+1})^2\overset{P}\longrightarrow V.
\end{equation*}

Next, define the $\mathcal{F}$-martingale
\begin{equation*}
L_0=Y_0,\quad
L_n=Y_n-\sum_{k=0}^{n-1}E\bigl(Y_{k+1}-Y_k\mid\mathcal{F}_k\bigr).
\end{equation*}
Define also $T_n=E\bigl(Y_{n+1}-Y_n\mid\mathcal{F}_n\bigr)$. By
\eqref{basicnew},
\begin{gather}\label{p087bnh65r}
\sqrt{n}\,\sum_{k\geq n}E\abs{I_{H_k}\,T_k}\leq\sqrt{n}\,\sum_{k\geq
n}\sqrt{E(I_{H_k}\,T_k^2)}=\sqrt{n}\,\sum_{k\geq
n}\text{o}(k^{-3/2})\longrightarrow 0.
\end{gather}
In particular, $\sum_{k=0}^\infty E\abs{I_{H_k}\,T_k}<\infty$ so
that $\sum_{k=0}^{n-1}I_{H_k}\,T_k$ converges a.s.. Since $Y_n$
converges a.s. and $P(I_{H_n}\neq 1$ i.o.$)=0$,
\begin{equation*}
L_n=Y_n-\sum_{k=0}^{n-1}T_k\overset{a.s.}\longrightarrow
L\quad\text{for some random variable }L.
\end{equation*}

Next, write
\begin{equation*}
(L_n-L)-(Y_n-Y)=\sum_{k\geq n}(L_k-L_{k+1})-\sum_{k\geq
n}(Y_k-Y_{k+1})=\sum_{k\geq n}T_k.
\end{equation*}
Recalling $\sqrt{n}\,\sum_{k\geq
n}\abs{I_{H_k}\,T_k}\overset{P}\longrightarrow 0$ (thanks to
\eqref{p087bnh65r}), one obtains
\begin{gather*}
\Abs{\sqrt{n}\,(L_n-L)-\sqrt{n}\,(Y_n-Y)}=\sqrt{n}\,\Abs{\sum_{k\geq
n}T_k}
\\\leq\sqrt{n}\,\sum_{k\geq
n}\abs{I_{H_k}\,T_k}+\sqrt{n}\,\sum_{k\geq
n}\abs{(1-I_{H_k})\,T_k}\overset{P}\longrightarrow 0.
\end{gather*}
Thus, it suffices to prove
$\sqrt{n}\,(L_n-L)\longrightarrow\mathcal{N}(0,V)$
$\mathcal{F}$-stably in strong sense, that is, to prove conditions
(i) and (ii). Condition (i) reduces to \eqref{f6y8j0w} after noting
that $L_k-L_{k+1}=E(Y_{k+1}\mid\mathcal{F}_k)-Y_{k+1}$.

As to (ii), since $L_k-L_{k+1}=Y_k-Y_{k+1}+T_k$, condition
\eqref{udlq2e} yields
\begin{gather*}
n\sum_{k\geq n}(L_k-L_{k+1})^2=V+\,n\sum_{k\geq
n}\bigl\{T_k^2+2\,T_k(Y_k-Y_{k+1})\bigr\}+\text{o}_P(1).
\end{gather*}
By \eqref{basicnew}, $E\bigl\{n\sum_{k\geq
n}I_{H_k}T_k^2\bigr\}=n\sum_{k\geq n}$o$(k^{-3})\longrightarrow 0$.
Since $P(I_{H_n}\neq 1$ i.o.$)=0$, then $n\sum_{k\geq
n}T_k^2\overset{P}\longrightarrow 0$. Because of \eqref{udlq2e},
this also implies
\begin{gather*}
\bigl\{\,n\sum_{k\geq n}T_k(Y_k-Y_{k+1})\bigr\}^2\leq n\sum_{k\geq
n}T_k^2\,\cdot\,n\sum_{k\geq
n}(Y_k-Y_{k+1})^2\overset{P}\longrightarrow 0.
\end{gather*}
Therefore, condition (ii) holds and this concludes the proof.
\end{proof}

We next turn to Theorem \ref{main}. From now on, it is assumed
$d_0<d$ (the case $d_0=d$ has been settled in \cite{BCPR}). Recall
the notations $S_n^*=\sum_{i=1}^{d_0}N_{n,i}$ and
$S_n=\sum_{i=1}^dN_{n,i}$. Note also that, by a straightforward
calculation,
\begin{gather*}
Z_{n+1,j}^*-Z_{n,j}^*=\frac{X_{n+1,j}\,A_{n+1,j}}{S_n^*+A_{n+1,j}}\,-\,Z_{n,j}^*\sum_{i=1}^{d_0}\frac{X_{n+1,i}\,A_{n+1,i}}{S_n^*+A_{n+1,i}}.
\end{gather*}

\begin{proof} [{\bf Proof of Theorem \ref{main}}] The proof is split
into two steps.

\vspace{0.2cm}

{\bf (i) $D_{n,j}^*\longrightarrow\mathcal{N}(0,V_j)$
$\mathcal{G}$-stably in strong sense.}

\noindent By Lemma \ref{mf65tr},
$Z_{n,j}^*=\frac{Z_{n,j}}{\sum_{i=1}^{d_0}Z_{n,i}}\overset{a.s.}\longrightarrow
Z_{(j)}$. Further, $P(2\,S_n^*< n\,m$ i.o.$)=0$ since
$\frac{S_n^*}{n}\overset{a.s.}\longrightarrow m$. Hence, by
Proposition \ref{beo5t}, it suffices to prove conditions
\eqref{basicnew}-\eqref{f6y8j0w}-\eqref{udlq2e} with
\begin{equation*}
\mathcal{F}_n=\mathcal{G}_n,\quad Y_n=Z_{n,j}^*,\quad
Y=Z_{(j)},\quad H_n=\{2\,S_n^*\geq n\,m\},\quad V=V_j.
\end{equation*}

Conditions \eqref{basicnew} and \eqref{f6y8j0w} trivially hold. As
to \eqref{basicnew}, note that
\begin{equation*}
Z_{n,j}^*\sum_{i=1}^{d_0}Z_{n,i}=Z_{n,j}\sum_{i=1}^{d_0}Z_{n,i}^*=Z_{n,j}.
\end{equation*}
Therefore,
\begin{gather*}
E\bigl\{Z_{n+1,j}^*-Z_{n,j}^*\mid\mathcal{G}_n\bigr\}=Z_{n,j}\,E\bigl\{\frac{A_{n+1,j}}{S_n^*+A_{n+1,j}}\mid\mathcal{G}_n\bigr\}\,-\,Z_{n,j}^*\sum_{i=1}^{d_0}Z_{n,i}\,E\bigl\{\frac{A_{n+1,i}}{S_n^*+A_{n+1,i}}\mid\mathcal{G}_n\bigr\}
\\=-Z_{n,j}\,E\bigl\{\frac{A_{n+1,j}^2}{S_n^*(S_n^*+A_{n+1,j})}\mid\mathcal{G}_n\bigr\}\,+\,Z_{n,j}^*\sum_{i=1}^{d_0}Z_{n,i}\,E\bigl\{\frac{A_{n+1,i}^2}{S_n^*(S_n^*+A_{n+1,i})}\mid\mathcal{G}_n\bigr\},
\\\text{so that }\,\,I_{H_n}\,\Abs{\,E\bigl\{Z_{n+1,j}^*-Z_{n,j}^*\mid\mathcal{G}_n\bigr\}}\leq
I_{H_n}\,\frac{d_0\,\beta^2}{(S_n^*)^2}\leq\frac{4\,d_0\,\beta^2}{m^2}\,\frac{1}{n^2}.
\end{gather*}
As to \eqref{f6y8j0w},
\begin{gather*}
\Abs{\,E\bigl(Z_{k+1,j}^*\mid\mathcal{G}_k\bigr)-Z_{k+1,j}^*}\leq\frac{2\,\beta}{S_k^*}\,+\,N_{k,j}\,\Abs{E\bigl(\frac{1}{S_{k+1}^*}\mid\mathcal{G}_k\bigr)-\frac{1}{S_{k+1}^*}}
\\\leq\frac{2\,\beta}{S_k^*}\,+\,N_{k,j}\,\bigl(\frac{1}{S_k^*}-\frac{1}{S_k^*+\beta}\bigr)\leq\frac{3\,\beta}{S_k^*},
\\\text{so that }\,\,I_{H_n}\sup_{k\geq
n}\abs{E(Z_{k+1,j}^*\mid\mathcal{G}_k)-Z_{k+1,j}^*}\leq
I_{H_n}\,\frac{3\,\beta}{S_n^*}\leq\frac{6\,\beta}{m}\,\frac{1}{n}.
\end{gather*}

Finally, let us turn to \eqref{udlq2e}. For every
$i\in\{1,\ldots,d_0\}$,
\begin{gather*}
n^2E\bigl\{\frac{A_{n+1,i}^2}{(S_n^*+A_{n+1,i})^2}\mid\mathcal{G}_n\bigr\}\leq
n^2\frac{EA_{n+1,i}^2}{(S_n^*)^2}\,\overset{a.s.}\longrightarrow\,\frac{q_i}{m^2}\,\,\text{
and}
\\n^2E\bigl\{\frac{A_{n+1,i}^2}{(S_n^*+A_{n+1,i})^2}\mid\mathcal{G}_n\bigr\}\geq
n^2\frac{EA_{n+1,i}^2}{(S_n^*+\beta)^2}\,\overset{a.s.}\longrightarrow\,\frac{q_i}{m^2}.
\end{gather*}
Since $X_{n+1,r}\,X_{n+1,s}=0$ for $r\neq s$, it follows that
\begin{gather*}
n^2E\bigl\{(Z_{n+1,j}^*-Z_{n,j}^*)^2\mid\mathcal{G}_n\bigr\}=n^2Z_{n,j}\,(1-Z_{n,j}^*)^2E\bigl\{\frac{A_{n+1,j}^2}{(S_n^*+A_{n+1,j})^2}\mid\mathcal{G}_n\bigr\}\,+
\\+\,n^2(Z_{n,j}^*)^2\sum_{i\leq d_0,i\neq j}Z_{n,i}\,E\bigl\{\frac{A_{n+1,i}^2}{(S_n^*+A_{n+1,i})^2}\mid\mathcal{G}_n\bigr\}
\\\overset{a.s.}\longrightarrow
Z_{(j)}(1-Z_{(j)})^2\frac{q_j}{m^2}\,+\,Z_{(j)}^2\sum_{i\leq
d_0,i\neq j}Z_{(i)}\frac{q_i}{m^2}=V_j.
\end{gather*}
Let $R_{n+1}=(n+1)^2I_{H_n}\,(Z_{n+1,j}^*-Z_{n,j}^*)^2$. Since
$H_n\in\mathcal{G}_n$ and $P(I_{H_n}\neq 1$ i.o.$)=0$, then
$E(R_{n+1}\mid\mathcal{G}_n)\overset{a.s.}\longrightarrow V_j$. On
noting that
$\abs{Z_{n+1,j}^*-Z_{n,j}^*}\leq\frac{d_0\,\beta}{S_n^*}$,
\begin{equation*}
\frac{ER_n^2}{n^2}\leq
(d_0\,\beta)^4n^2E\bigl(\frac{I_{H_{n-1}}}{(S_{n-1}^*)^4}\bigr)\leq\bigl(\frac{2\,d_0\,\beta}{m}\bigr)^4\frac{n^2}{(n-1)^4}.
\end{equation*}
By Lemma \ref{hftuimn} (applied with $Y_n=R_n$),
\begin{equation*}
n\sum_{k\geq
n}I_{H_k}(Z_{k+1,j}^*-Z_{k,j}^*)^2=\frac{n}{n+1}\,(n+1)\sum_{k\geq
n+1}\frac{R_k}{k^2}\,\overset{a.s.}\longrightarrow\, V_j.
\end{equation*}
Since $P(I_{H_n}\neq 1$ i.o.$)=0$ then $n\sum_{k\geq
n}(Z_{k+1,j}^*-Z_{k,j}^*)^2\overset{a.s.}\longrightarrow\, V_j$,
that is, condition \eqref{udlq2e} holds.

\vspace{0.2cm}

{\bf (ii) $C_{n,j}^*\longrightarrow\mathcal{N}(0,U_j)$ stably.}

\vspace{0.2cm}

\noindent Define $T_{n,i}=\sum_{k=1}^nX_{k,i}$, $T_{0,i}=0$, and
note that
\begin{gather*}
C_{n,j}^*=-\frac{\sqrt{n}\,Z_{n,j}^*}{1+\sum_{i=1}^{d_0}T_{n,i}}+\frac{n}{1+\sum_{i=1}^{d_0}T_{n,i}}\,\frac{T_{n,j}-Z_{n,j}^*\sum_{i=1}^{d_0}T_{n,i}}{\sqrt{n}}\quad\text{and}
\\T_{n,j}-Z_{n,j}^*\sum_{i=1}^{d_0}T_{n,i}=\sum_{k=1}^n\bigl\{X_{k,j}-Z_{k,j}^*\sum_{i=1}^{d_0}T_{k,i}+Z_{k-1,j}^*\sum_{i=1}^{d_0}T_{k-1,i}\bigr\}
\\=\sum_{k=1}^n\bigl\{X_{k,j}-Z_{k-1,j}^*\sum_{i=1}^{d_0}X_{k,i}-\sum_{i=1}^{d_0}T_{k,i}(Z_{k,j}^*-Z_{k-1,j}^*)\bigr\}.
\end{gather*}
Define also $H_n=\{2S_n^*\geq n\,m\}$ and
\begin{gather*}
C_{n,j}^{**}=\frac{1}{\sqrt{n}}\,\sum_{k=1}^nI_{H_{k-1}}\bigl\{X_{k,j}-Z_{k-1,j}^*\sum_{i=1}^{d_0}X_{k,i}+\sum_{i=1}^{d_0}T_{k-1,i}\,\bigl(E(Z_{k,j}^*\mid\mathcal{G}_{k-1})-Z_{k,j}^*\bigr)\bigr\}.
\end{gather*}
Recalling (from point (i)) that $P(I_{H_n}\neq 1$ i.o.$)=0$,
$\lim_n\frac{\sum_{i=1}^{d_0}T_{n,i}}{n}=1$ a.s., and
$I_{H_{k-1}}\Abs{\,E\bigl\{Z_{k,j}^*-Z_{k-1,j}^*\mid\mathcal{G}_{k-1}\bigr\}}\leq\frac{c}{(k-1)^2}$
a.s. for some constant $c$, it is not hard to see that
$C_{n,j}^*\longrightarrow N$ stably if and only if
$C_{n,j}^{**}\longrightarrow N$ stably for any kernel $N$.

We next prove $C_{n,j}^{**}\longrightarrow\mathcal{N}(0,U_j)$
stably. For $k=1,\ldots,n$, let $\mathcal{F}_{n,k}=\mathcal{G}_k$
and

\begin{equation*}
Y_{n,k}=\frac{I_{H_{k-1}}\bigl\{X_{k,j}\,-\,Z_{k-1,j}^*\sum_{i=1}^{d_0}X_{k,i}+\sum_{i=1}^{d_0}T_{k-1,i}\,(E(Z_{k,j}^*\mid\mathcal{G}_{k-1})-Z_{k,j}^*)\,\bigl\}}{\sqrt{n}}.
\end{equation*}
Since $E(Y_{n,k}\mid\mathcal{F}_{n,k-1})=0$ a.s., the martingale CLT
(see Theorem 3.2 of \cite{HH}) applies. As a consequence,
$C_{n,j}^{**}=\sum_{k=1}^nY_{n,k}\longrightarrow\mathcal{N}(0,U_j)$
stably provided
\begin{equation*}
\sup_nE\bigl(\max_{1\leq k\leq
n}Y_{n,k}^2\bigr)<\infty;\,\,\,\max_{1\leq k\leq
n}\abs{Y_{n,k}}\overset{P}\longrightarrow
0;\,\,\,\sum_{k=1}^nY_{n,k}^2\overset{P}\longrightarrow U_j.
\end{equation*}
As shown in point (i),
$I_{H_{k-1}}\Abs{E(Z_{k,j}^*\mid\mathcal{G}_{k-1})-Z_{k,j}^*}\leq\frac{d}{k-1}$
a.s. for a suitable constant $d$. Hence, the first two conditions
follow from
\begin{equation*}
Y_{n,k}^2\leq\frac{2}{n}+\frac{2}{n}I_{H_{k-1}}(k-1)^2\bigl(E(Z_{k,j}^*\mid\mathcal{G}_{k-1})-Z_{k,j}^*\bigr)^2
\leq\frac{2\,(1+d^2)}{n}\quad\text{a.s.}.
\end{equation*}

To conclude the proof, it remains to see that
$\sum_{k=1}^nY_{n,k}^2\overset{P}\longrightarrow U_j$. After some
(long) algebra, the latter condition is shown equivalent to
\begin{equation}\label{mal12ed}
\frac{1}{n}\,\sum_{k=1}^nI_{H_{k-1}}\bigl\{X_{k,j}-Z_{k-1,j}^*+k\,\bigl(Z_{k-1,j}^*-Z_{k,j}^*\bigr)\bigr\}^2\overset{P}\longrightarrow
U_j.
\end{equation}
Let $R_{n+1}=(n+1)^2I_{H_n}\,(Z_{n+1,j}^*-Z_{n,j}^*)^2$. Since
$E(R_{n+1}\mid\mathcal{G}_n)\overset{a.s.}\longrightarrow V_j$, as
shown in point (i), Lemma \ref{hftuimn} implies
\begin{gather*}
\frac{1}{n}\,\sum_{k=1}^nI_{H_{k-1}}k^2\bigl(Z_{k-1,j}^*-Z_{k,j}^*\bigr)^2\overset{a.s.}\longrightarrow
V_j.
\end{gather*}
A direct calculation shows that
\begin{gather*}
\frac{1}{n}\,\sum_{k=1}^nI_{H_{k-1}}(X_{k,j}-Z_{k-1,j}^*)^2\overset{a.s.}\longrightarrow
Z_{(j)}(1-Z_{(j)}). \end{gather*} Finally, observe the following
facts
\begin{gather*}
\bigl(Z_{n,j}^*-Z_{n+1,j}^*\bigr)\,(X_{n+1,j}-Z_{n,j}^*)
=-(1-Z_{n,j}^*)\,\frac{X_{n+1,j}\,A_{n+1,j}}{S_n^*+A_{n+1,j}}-Z_{n,j}^*(Z_{n,j}^*-Z_{n+1,j}^*),
\\(n+1)\,Z_{n,j}^*\,I_{H_n}\,\Abs{E\bigl(Z_{n,j}^*-Z_{n+1,j}^*\mid\mathcal{G}_n\bigr)}\leq\frac{c\,(n+1)}{n^2}\overset{a.s.}\longrightarrow
0,
\\(n+1)\,E\bigl\{\frac{X_{n+1,j}\,A_{n+1,j}}{S_n^*+A_{n+1,j}}\mid\mathcal{G}_n\bigr\}\leq\frac{n+1}{S_n^*}\,Z_{n,j}\,EA_{n+1,j}\overset{a.s.}\longrightarrow
Z_{(j)},
\\(n+1)\,E\bigl\{\frac{X_{n+1,j}\,A_{n+1,j}}{S_n^*+A_{n+1,j}}\mid\mathcal{G}_n\bigr\}\geq\frac{n+1}{S_n^*+\beta}\,Z_{n,j}\,EA_{n+1,j}\overset{a.s.}\longrightarrow
Z_{(j)}.
\end{gather*}
Therefore,
\begin{gather*}
(n+1)\,I_{H_n}E\bigl\{\,(Z_{n,j}^*-Z_{n+1,j}^*)\,(X_{n+1,j}-Z_{n,j}^*)\mid\mathcal{G}_n\bigr\}
\overset{a.s.}\longrightarrow -Z_{(j)}(1-Z_{(j)})
\end{gather*}
and Lemma \ref{hftuimn} again implies
\begin{gather*}
\frac{2}{n}\,\sum_{k=1}^nI_{H_{k-1}}k\,(Z_{k-1,j}^*-Z_{k,j}^*)\,(X_{k,j}-Z_{k-1,j}^*)\overset{a.s.}\longrightarrow
-2\,Z_{(j)}(1-Z_{(j)}). \end{gather*} Thus condition \eqref{mal12ed}
holds, and this concludes the proof.

\end{proof}

\begin{rem}\label{h6t98d34s} Point (ii) admits a simpler proof in case $EA_{k,j}=m$
for all $k\geq 1$ and $1\leq j\leq d_0$. This happens, in
particular, if the sequence $(A_{n,1},\ldots,A_{n,d})$ is i.i.d..

Given the real numbers $b_1,\ldots,b_{d_0}$, define
\begin{equation*}
Y_{n,k}=\frac{1}{\sqrt{n}}\,\sum_{j=1}^{d_0}b_j\,X_{k,j}\,(A_{k,j}-EA_{k,j}),\quad\mathcal{F}_{n,k}=\mathcal{G}_k,\quad
k=1,\ldots,n.
\end{equation*}
By Lemma \ref{hftuimn},
$\sum_{k=1}^nY_{n,k}^2\overset{a.s.}\longrightarrow\sum_{j=1}^{d_0}b_j^2\,(q_j-m^2)\,Z_{(j)}:=L$.
Thus, the martingale CLT implies
$\sum_{k=1}^nY_{n,k}\longrightarrow\mathcal{N}(0,L)$ stably. Since
$b_1,\ldots,b_{d_0}$ are arbitrary constants,
\begin{equation*}
\Bigl(\,\frac{\sum_{k=1}^nX_{k,j}\,(A_{k,j}-EA_{k,j})}{\sqrt{n}}\,:\,j=1,\ldots,d_0\Bigr)\longrightarrow
\mathcal{N}_{d_0}(0,\Sigma)\quad\text{stably}
\end{equation*}
where $\Sigma$ is the diagonal matrix with
$\sigma_{j,j}=(q_j-m^2)Z_{(j)}$. Let $T_{n,j}=\sum_{k=1}^nX_{k,j}$.
Since $EA_{k,j}=m$ and
$\frac{T_{n,j}}{n}\overset{a.s.}\longrightarrow Z_{(j)}>0$ for all
$j\leq d_0$, one also obtains
\begin{equation*}
\Bigl(\,\sqrt{n}\,\,\bigl\{\frac{\sum_{k=1}^nX_{k,j}\,A_{k,j}}{T_{n,j}}-m\bigr\}\,:\,j=1,\ldots,d_0\Bigr)\longrightarrow
\mathcal{N}_{d_0}(0,\Gamma)\quad\text{stably}
\end{equation*}
where $\Gamma$ is diagonal with
$\gamma_{j,j}=\frac{(q_j-m^2)}{Z_{(j)}}$. Next, write
\begin{gather*}
\widetilde{C}_{n,j}:=\sqrt{n}\,\bigl(\frac{T_{n,j}}{\sum_{i=1}^{d_0}T_{n,i}}-\frac{\sum_{k=1}^nX_{k,j}A_{k,j}}{\sum_{i=1}^{d_0}\sum_{k=1}^nX_{k,i}A_{k,i}}\bigr)
\\=\frac{T_{n,j}}{\sum_{i=1}^{d_0}\sum_{k=1}^nX_{k,i}A_{k,i}}\,\frac{\sum_{i\leq d_0,i\neq
j}T_{n,i}}{\sum_{i=1}^{d_0}T_{n,i}}\,\sqrt{n}\,\bigl(m-\frac{\sum_{k=1}^nX_{k,j}A_{k,j}}{T_{n,j}}\,\bigr)\,+
\\+\,\frac{T_{n,j}}{\sum_{i=1}^{d_0}\sum_{k=1}^nX_{k,i}A_{k,i}}\,\frac{1}{\sum_{i=1}^{d_0}T_{n,i}}\,\sum_{i\leq d_0,i\neq
j}T_{n,i}\,\sqrt{n}\,\bigl(\frac{\sum_{k=1}^nX_{k,i}A_{k,i}}{T_{n,i}}-m\bigr).
\end{gather*}
Clearly, $C_{n,j}^*-\widetilde{C}_{n,j}\overset{a.s.}\longrightarrow
0$. To conclude the proof, it suffices noting that
$\widetilde{C}_{n,j}$ converges stably to the Gaussian kernel with
mean 0 and variance
\begin{gather*}
\Bigl(\frac{Z_{(j)}(1-Z_{(j)})}{m}\Bigr)^2\,\frac{q_j-m^2}{Z_{(j)}}\,+\,\frac{Z_{(j)}^2}{m^2}\sum_{i\leq
d_0,i\neq j}Z_{(i)}^2\,\frac{q_i-m^2}{Z_{(i)}}=U_j.
\end{gather*}

\end{rem}

\vspace{0.1cm}

\begin{center}{\bf APPENDIX}\end{center}

\vspace{0.1cm}

\begin{proof} [{\bf Proof of Lemma \ref{mf65tr}}]
We first note that $N_{n,j}\overset{a.s.}\longrightarrow\infty$ for
each $j\leq d_0$. Arguing as in the proof of Proposition 2.3 of
\cite{MF}, in fact, $\sum_{n=1}^\infty X_{n,j}=\infty$ a.s.. Hence,
$\sum_{k=1}^nX_{k,j}\,EA_{k,j}\overset{a.s}\longrightarrow\infty$,
and $N_{n,j}\overset{a.s.}\longrightarrow\infty$ follows from
\begin{equation*}
L_n=N_{n,j}-\bigl\{a_j+\sum_{k=1}^nX_{k,j}\,EA_{k,j}\bigr\}=\sum_{k=1}^nX_{k,j}\,\bigl(A_{k,j}-EA_{k,j}\bigr)
\end{equation*}
is a $\mathcal{G}$-martingale such that $\abs{L_{n+1}-L_n}\leq\beta$
for all $n$.

We also need the following fact.

\vspace{0.2cm}

{\em CLAIM:}  $\tau_{n,j}=\frac{N_{n,j}}{(S_n^*)^\lambda}\,$
converges a.s. for all $j>d_0$ and $\lambda\in
(\frac{\lambda_0}{m},\,1)$.

On noting that $(1-x)^\lambda\leq1-\lambda\,x$ for $0\leq x\leq 1$
and $\sum_{i=1}^{d_0}Z_{n,i}=\frac{S_n^*}{S_n}$ , one can estimate
as follows
\begin{gather*}
E\bigl\{\frac{\tau_{n+1,j}}{\tau_{n,j}}-1\mid\mathcal{G}_n\bigr\}=E\bigl\{\frac{N_{n,j}+X_{n+1,j}\,A_{n+1,j}}{N_{n,j}}\,(\frac{S_n^*}{S_{n+1}^*})^\lambda\mid\mathcal{G}_n\bigr\}-1
\\\leq\frac{Z_{n,j}\,EA_{n+1,j}}{N_{n,j}}\,+\,E\bigl\{(\frac{S_n^*}{S_{n+1}^*})^\lambda-1\mid\mathcal{G}_n\bigr\}
\\\leq\frac{EA_{n+1,j}}{S_n}\,-\,\lambda\,\sum_{i=1}^{d_0}E\bigl\{\frac{X_{n+1,i}\,A_{n+1,i}}{S_{n+1}^*}\mid\mathcal{G}_n\bigr\}
\\\leq\frac{EA_{n+1,j}}{S_n}\,-\,\lambda\,\sum_{i=1}^{d_0}\frac{Z_{n,i}\,EA_{n+1,i}}{S_n^*+\beta}
\\=\frac{EA_{n+1,j}}{S_n}\,-\,\lambda\,EA_{n+1,1}\frac{S_n^*}{S_n(S_n^*+\beta)}
\\=\frac{1}{S_n}\,\bigl(EA_{n+1,j}\,-\,\lambda\,EA_{n+1,1}\frac{S_n^*}{S_n^*+\beta}\bigr)\quad\text{a.s.}.
\end{gather*}
Since
$\limsup_n\bigl(EA_{n+1,j}-\lambda\,EA_{n+1,1}\bigr)\leq\lambda_0-\lambda\,m<0$,
there are $\epsilon>0$ and $n_0\geq 1$ such that
$EA_{n+1,j}-\lambda\,EA_{n+1,1}\leq-\epsilon$ whenever $n\geq n_0$.
Thus,
\begin{equation*}
E\bigl\{\tau_{n+1,j}-\tau_{n,j}\mid\mathcal{G}_n\bigr\}=\tau_{n,j}\,E\bigl\{\frac{\tau_{n+1,j}}{\tau_{n,j}}-1\mid\mathcal{G}_n\bigr\}\leq
0\quad\text{a.s. whenever }n\geq n_0\text{ and }S_n^*\geq c
\end{equation*}
for a suitable constant $c$. Since $S_n^*\geq
N_{n,1}\overset{a.s.}\longrightarrow\infty$, thus, $(\tau_{n,j})$ is
eventually a non negative $\mathcal{G}$-super-martingale. Hence,
$\tau_{n,j}$ converges a.s..

Let $\lambda\in (\frac{\lambda_0}{m},\,1)$. A first consequence of
the Claim is that
$Z_{n,j}\leq\frac{\tau_{n,j}}{S_n^{1-\lambda}}\overset{a.s.}\longrightarrow
0$ for each $j>d_0$. Letting $Y_n=\sum_{i=1}^{d_0}X_{n,i}\,A_{n,i}$,
this implies
\begin{equation*}
E(Y_{n+1}\mid\mathcal{G}_n)=\sum_{i=1}^{d_0}Z_{n,i}\,EA_{n+1,i}=EA_{n+1,1}\,(1-\sum_{i=d_0+1}^dZ_{n,i})\overset{a.s.}\longrightarrow
m.
\end{equation*}
Thus, Lemma \ref{hftuimn} yields
$\frac{S_n^*}{n}\overset{a.s.}\longrightarrow m$. Similarly,
$\frac{S_n}{n}\overset{a.s.}\longrightarrow m$. Applying the Claim
again,
\begin{equation*}
n^{1-\lambda}Z_{n,j}=(\frac{n}{S_n})^{1-\lambda}\,(\frac{S_n^*}{S_n})^\lambda\,\tau_{n,j}\quad\text{converges
a.s. for each }j>d_0.
\end{equation*}
Since $j>d_0$ and $\lambda\in (\frac{\lambda_0}{m},\,1)$ are
arbitrary, it follows that
$n^{1-\lambda}\sum_{j=d_0+1}^dZ_{n,j}\overset{a.s.}\longrightarrow
0$ for each $\lambda>\frac{\lambda_0}{m}$.

Next, fix $j\leq d_0$. For $Z_{n,j}$ to converge a.s., it suffices
that
\begin{equation*}
\sum_nE\bigl\{Z_{n+1,j}-Z_{n,j}\mid\mathcal{G}_n\bigr\}\,\text{ and
}\,
\sum_nE\bigl\{(Z_{n+1,j}-Z_{n,j})^2\mid\mathcal{G}_n\bigr\}\quad\text{converge
a.s.};
\end{equation*}
see Lemma 3.2 of \cite{PV}. Since
\begin{gather*}
Z_{n+1,j}-Z_{n,j}=\frac{X_{n+1,j}\,A_{n+1,j}}{S_n+A_{n+1,j}}\,-\,Z_{n,j}\sum_{i=1}^d\frac{X_{n+1,i}\,A_{n+1,i}}{S_n+A_{n+1,i}},
\end{gather*}
then $\abs{Z_{n+1,j}-Z_{n,j}}\leq\frac{d\,\beta}{S_n}$. Hence,
\begin{equation*}
\sum_nE\bigl\{(Z_{n+1,j}-Z_{n,j})^2\mid\mathcal{G}_n\bigr\}\leq
d^2\beta^2\sum_n\frac{1}{n^2}\,(\frac{n}{S_n})^2<\infty\quad\text{a.s.}.
\end{equation*}
Moreover,
\begin{gather*}
E\bigl\{Z_{n+1,j}-Z_{n,j}\mid\mathcal{G}_n\bigr\}=Z_{n,j}\,E\bigl\{\frac{A_{n+1,j}}{S_n+A_{n+1,j}}\mid\mathcal{G}_n\bigr\}
\,-\,Z_{n,j}\sum_{i=1}^dZ_{n,i}\,E\bigl\{\frac{A_{n+1,i}}{S_n+A_{n+1,i}}\mid\mathcal{G}_n\bigr\}
\\=-Z_{n,j}\,E\bigl\{\frac{A_{n+1,j}^2}{S_n(S_n+A_{n+1,j})}\mid\mathcal{G}_n\bigr\}
\,+\,Z_{n,j}\sum_{i=1}^dZ_{n,i}\,E\bigl\{\frac{A_{n+1,i}^2}{S_n(S_n+A_{n+1,i})}\mid\mathcal{G}_n\bigr\}\,+
\\+\,Z_{n,j}\,\frac{EA_{n+1,j}}{S_n}\,-\,Z_{n,j}\sum_{i=1}^dZ_{n,i}\,\frac{EA_{n+1,i}}{S_n}\quad\text{a.s., and}
\\EA_{n+1,j}\,-\,\sum_{i=1}^dZ_{n,i}\,EA_{n+1,i}=
EA_{n+1,1}\sum_{i=d_0+1}^{d}Z_{n,i}\,-\sum_{i=d_0+1}^{d}Z_{n,i}EA_{n+1,i}.
\end{gather*}
Therefore, $\sum_nE\bigl\{Z_{n+1,j}-Z_{n,j}\mid\mathcal{G}_n\bigr\}$
converges a.s. since
\begin{gather*}
\Abs{E\bigl\{Z_{n+1,j}-Z_{n,j}\mid\mathcal{G}_n\bigr\}}\leq
\frac{d\,\beta^2}{S_n^2}+2\,\beta\,\frac{\sum_{i=d_0+1}^dZ_{n,i}}{S_n}
=\,\text{o}(n^{\lambda-2})\quad\text{a.s. for each }\lambda\in
(\frac{\lambda_0}{m},\,1).
\end{gather*}

Thus, $Z_{n,j}\overset{a.s.}\longrightarrow Z_{(j)}$ for some random
variable $Z_{(j)}$. To conclude the proof, we let
$Y_{n,i}=\log\frac{Z_{n,i}}{Z_{n,1}}$ and prove that
\begin{gather*}
\sum_nE\bigl\{Y_{n+1,i}-Y_{n,i}\mid\mathcal{G}_n\bigr\}\,\text{ and
}\,
\sum_nE\bigl\{(Y_{n+1,i}-Y_{n,i})^2\mid\mathcal{G}_n\bigr\}\text{
converge a.s. whenever }i\leq d_0.
\end{gather*} In this case, in fact, $\log\frac{Z_{n,i}}{Z_{n,1}}$
converges a.s. for each $i\leq d_0$ and this implies $Z_{(i)}>0$
a.s. for each $i\leq d_0$.

Since
$Y_{n+1,i}-Y_{n,i}=X_{n+1,i}\log\bigl(1+\frac{A_{n+1,i}}{N_{n,i}}\bigr)-X_{n+1,1}\log\bigl(1+\frac{A_{n+1,1}}{N_{n,1}}\bigr)$,
then
\begin{equation*}
E\bigl\{Y_{n+1,i}-Y_{n,i}\mid\mathcal{G}_n\bigr\}=Z_{n,i}E\bigl\{\log\bigl(1+\frac{A_{n+1,i}}{N_{n,i}}\bigr)\mid\mathcal{G}_n\bigr\}
-Z_{n,1}E\bigl\{\log\bigl(1+\frac{A_{n+1,1}}{N_{n,1}}\bigr)\mid\mathcal{G}_n\bigr\}\quad\text{a.s..}
\end{equation*}
Since $EA_{n+1,i}=EA_{n+1,1}$, a second order Taylor expansion of
$x\mapsto\log(1+x)$ yields
\begin{equation*}
\Abs{E\bigl\{Y_{n+1,i}-Y_{n,i}\mid\mathcal{G}_n\bigr\}}\leq\frac{\beta^2}{S_n}\,\bigl(\frac{1}{N_{n,i}}+\frac{1}{N_{n,1}}\bigr)
\quad\text{a.s.}.
\end{equation*}
A quite similar estimate holds for
$E\bigl\{(Y_{n+1,i}-Y_{n,i})^2\mid\mathcal{G}_n\bigr\}$. Thus, it
suffices to see
\begin{equation*}
\sum_n\frac{1}{S_n\,N_{n,i}}<\infty\quad\text{a.s. for each }i\leq
d_0.
\end{equation*}
Define $R_{n,i}=\frac{(S_n^*)^u}{N_{n,i}}$ where $u\in (0,1)$ and
$i\leq d_0$. Since $(1+x)^u\leq 1+u\,x$ for $x\geq 0$, one can
estimate as
\begin{gather*}
E\bigl\{\frac{R_{n+1,i}}{R_{n,i}}-1\mid\mathcal{G}_n\bigr\}=E\bigl\{(\frac{S_{n+1}^*}{S_n^*})^u-1\mid\mathcal{G}_n\bigr\}
-E\bigl\{(\frac{S_{n+1}^*}{S_n^*})^u\,\frac{X_{n+1,i}\,A_{n+1,i}}{N_{n,i}+A_{n+1,i}}\mid\mathcal{G}_n\bigr\}
\\\leq u\,E\bigl\{\frac{S_{n+1}^*-S_n^*}{S_n^*}\mid\mathcal{G}_n\bigr\}
-E\bigl\{\frac{X_{n+1,i}\,A_{n+1,i}}{N_{n,i}+\beta}\mid\mathcal{G}_n\bigr\}
\\=\frac{u}{S_n^*}\,\sum_{p=1}^{d_0}Z_{n,p}\,EA_{n+1,p}
-\frac{Z_{n,i}EA_{n+1,i}}{N_{n,i}+\beta}
=\frac{EA_{n+1,1}}{S_n}\,\bigl\{u-\frac{N_{n,i}}{N_{n,i}+\beta}\bigr\}\,\text{
a.s..}
\end{gather*}
As in the proof of the Claim, thus,
\begin{equation*}
E\bigl\{R_{n+1,i}-R_{n,i}\mid\mathcal{G}_n\bigr\}=R_{n,i}\,E\bigl\{\frac{R_{n+1,i}}{R_{n,i}}-1\mid\mathcal{G}_n\bigr\}\leq
0\quad\text{a.s. whenever }N_{n,i}\geq c
\end{equation*}
for a suitable constant $c$. Since
$N_{n,i}\overset{a.s.}\longrightarrow\infty$, then $(R_{n,i})$ is
eventually a non negative $\mathcal{G}$-super-martingale, so that
$R_{n,i}$ converges a.s.. Hence,
\begin{equation*}
\sum_n\frac{1}{S_n\,N_{n,i}}=\sum_n\frac{R_{n,i}}{S_n\,(S_n^*)^u}=\sum_nR_{n,i}\,\frac{n}{S_n}\,(\frac{n}{S_n^*})^u\,\frac{1}{n^{1+u}}<\infty\quad\text{a.s..}
\end{equation*}
This concludes the proof.

\end{proof}

\end{document}